\documentclass[a4paper,reqno,12pt]{amsart}
\usepackage{t1enc}
\usepackage[margin=1in]{geometry}
\usepackage{ifthen}
\usepackage{graphicx}
\usepackage{xcolor}
\usepackage[normalem]{ulem}

\newcommand{\ind}{\mathbf{1}}

\theoremstyle{plain}
\newtheorem{theorem}{Theorem}
\newtheorem{lemma}{Lemma}
\newtheorem{corollary}{Corollary}
\newtheorem{proposition}{Proposition}

\theoremstyle{definition}

\theoremstyle{remark}

\def\p{\mathbb{P}}
\def\e{\mathbb{E}}
\def\ind{\mbox{\rm 1\hspace{-0.04in}I}}

\newcommand{\formula}[2][nolabel]
{\ifthenelse{\equal{#1}{nolabel}}
 {\begin{align*} #2 \end{align*}}
 {\ifthenelse{\equal{#1}{}}
  {\begin{align} #2 \end{align}}
  {\begin{align} \label{#1} #2 \end{align}}
 }
}

\numberwithin{equation}{section}

\title[Density of the supremum of L\'evy processes]
{The asymptotic behavior of the density of the supremum of L\'evy processes}

\author{Lo\"i{}c Chaumont and Jacek Ma{\l}ecki}

\address{Lo\"i{}c Chaumont -- LAREMA UMR CNRS 6093, Universit\'e d'Angers, 2, Bd Lavoisier \\
Angers Cedex 01, 49045, France}
\email{loic.chaumont@univ-angers.fr}
\address{Jacek Ma{\l}ecki  --  Institute of Mathematics and Computer Science, Wroc{\l}aw University of
Technology, Wybrze\.ze Wyspia\'nskiego 27, 50-370 Wroc{\l}aw, Poland.}
\email{jacek.malecki@pwr.wroc.pl}


\thanks{Ce travail a b\'en\'efici\'e d'une aide de l'Agence Nationale de la Recherche portant la r\'ef\'erence ANR-09-BLAN-
0084-01. Jacek Ma\l{}ecki was also supported by NCN grant no. 2011/03/D/ST1/00311.}

\begin{document}

\begin{abstract} Let us consider a real L\'evy process $X$ whose transition probabilities are absolutely continuous and have bounded densities. 
Then the law of the past supremum of $X$ before any deterministic time $t$ is absolutely continuous on $(0,\infty)$. We show that its density $f_t(x)$
is continuous on $(0,\infty)$ if and only if the potential density $h'$ of the upward ladder height process is continuous on $(0,\infty)$. Then we prove that $f_t$ 
behaves at 0 as $h'$.  We also describe the asymptotic behaviour of $f_t$, when $t$ tends to infinity. Then some related results are obtained 
for the  density of the meander and this  of the entrance law of the L\'evy process  conditioned to stay positive. 
\end{abstract}

\maketitle

\section{Introduction}

Since the work by Paul L\'evy \cite{le} for standard Brownian motion, the study of the law of the past supremum before a deterministic time 
of real L\'evy processes has given rise to a significant literature. This is mainly justified by the important number of applications of this functional in various 
domains such as risk and queuing theories but properties of its law may also  be useful for theoretical purposes. It is constantly involved in fluctuation 
theory,  for instance. 

Let us denote by $\overline{X}_t=\sup_{s\le t}X_s$ the past supremum at time $t>0$ of the real L\'evy process $X$. Recently in \cite{bib:kmr12} the 
asymptotic behaviour of the distribution function $\p(\overline{X}_t\le x)$  was deeply investigated and in \cite{bib:l12}, necessary and sufficient conditions 
where given for the law of  $\overline{X}_t$ to be absolutely continuous. A natural continuation of both these works consists in a detailed study 
of the density $f_t$ of this law, when it exists. For instance if the transition probabilities of the L\'evy process are absolutely continuous, then the law of the 
past supremum is absolutely continuous on $(0,\infty)$. In this paper, under the additional assumption that the transition densities of the L\'evy process are 
bounded, we show that $f_t$ is continuous at $x\in(0,\infty)$ if and only if  the potential density $h'$ of the upward ladder height process is continuous at this
point. Then, we describe the asymptotic behaviour of the density $f_t(x)$, when $x$ tends to 0.  This behaviour is the same as this of $h'$, up to a constant which 
is given by the tail distribution of the life time of the generic excursion of the L\'evy process reflected at its supremum. We also obtain some asymptotic results and 
estimates for $f_t$, when the time $t$ tends to infinity. We then observe that the behaviour of $f_t(x)$, when $x$ is small is the same as its  behaviour,
when $t$ tends to infinity. Most of the results displayed in this paper extend those obtained by Doney and  Savov in \cite{bib:ds10} for stable L\'evy processes.  

In the next section we recall some elements of excursion and fluctuation theory for L\'evy processes that are necessary for the proof of our main
results. In Section \ref{mainres}, we state the main results and Section \ref{proofs} is devoted to their proofs. The latter section as well as Section 
\ref{mainres} also contain some intermediary results on bridges, meanders and L\'evy processes conditioned to stay positive.

\section{Preliminaries}\label{prelim}
We denote by $\mathcal{D}$  the space of c\`{a}dl\`{a}g paths $\omega:[0,\infty )
\rightarrow \mathbb{R}\cup\{\infty\}$ with lifetime $\zeta(\omega )=\inf \{t\ge0:\omega _{t}=\infty\}$, with the usual convention that 
$\inf\emptyset=+\infty$. The space $\mathcal{D}$ is equipped with the Skorokhod topology,  its Borel $\sigma $-algebra $\mathcal{F}$, and the usual completed 
filtration $(\mathcal{F}_{s},s\geq 0)$ generated by the coordinate process $X=(X_{t},t\geq 0)$ on the space $\mathcal{D}$. We write  $\underline{X}$ and 
$\overline{X}$ for the infimum and supremum processes, that is
\[\underline{X}_{t}=\inf \{X_{s}:0\leq s\leq t\}\;\;\;\mbox{and}\;\;\;
\overline{X}_{t}=\sup \{X_{s}:0\leq s\leq t\}\,.\]
We also define the first passage time by $X$ in the open half line  $(-\infty,0)$ by:
\[\tau_0^-=\inf\{t>0:X_t<0\}\,.\]

We denote by $\mathbb{P}_x$ the law on $(\mathcal{D},\mathcal{F})$ of a L\'{e}vy process starting from $x\in\mathbb{R}$ and we will set 
$\mathbb{P}:=\mathbb{P}_0$. 
Define $X^*:=-X$, then the law of $X^*$ under $\p_x$ will be denoted by $\p^*_x$, that is $(X^*,\p_x)=(X,\p_x^*)$. We recall that the process $(X,\p_x^*)$ is
in duality with  $(X,\p)$, with respect to the Lebesgue measure. In this section, as well as in the biggest part of this paper, we make the following assumptions:
\[\left\{\begin{array}{ll}
(H_1)\;&\mbox{The transition semigroup of $(X,\p)$ is absolutely continuous and}\\
&\mbox{there is a version of its densities, denoted by $x\mapsto p_t(x)$, $x\in\mathbb{R}$, which}\\
&\mbox{are bounded for all $t>0$.}\\
(H_2)\;&\mbox{$(X,\p)$ is not a compound Poisson process and for all $c\ge0$, the process}\\
&\mbox{$((|X_t-ct|,t\ge0),\p)$ is not a subordinator.}
\end{array}\right.\]
Note that $(H_1)$ is equivalent to the apparently stronger condition saying that the characteristic function of $X$ is integrable for all $t>0$. Indeed, boundedness 
of $p_t$ implies that $p_t\in L^2(\mathbb{R})$ and consequently $e^{-t\Psi(\xi)}\in L^{2}(\mathbb{R})$, for all $t>0$ which implies that 
$e^{-t\Psi(\xi)}\in L^{1}(\mathbb{R})$, for all $t>0$.  Conversely, if $e^{-t\Psi(\xi)}\in L^{1}(\mathbb{R})$, for all $t>0$, then by the Riemann-Lebesgue lemma, $p_t\in\mathcal{C}_0(\mathbb{R})$, moreover the function $(t,x)\mapsto p_t(x)$ is jointly continuous on $(0,\infty)\times\mathbb{R}$. From a result in \cite{sh}, 
positivity of the density of the semigroup is ensured by conditions $(H_1)$ and $(H_2)$,  that is,
\begin{equation}\label{pt_positivity}
\mbox{$p_t(x)>0$, for all $t>0$ and $x\in\mathbb{R}$\,.}
\end{equation}
The latter is an essential property for our purpose. Actually compound Poisson processes are excluded here only because our setting is not adapted to their study. 
Note that assumptions $(H_1)$ and $(H_2)$ are  satisfied in many classical cases, such as stable processes or subordinated Brownian motions.

Recall that the reflected process $X-\underline{X}$  is Markovian and that under our assumptions, 0 is always regular for at least one of the half lines 
$(-\infty,0)$ or $(0,\infty)$. When 0 is regular for $(-\infty,0)$ (resp. $(0,\infty)$), we will simply say that $(-\infty,0)$ (resp. $(0,\infty)$) is regular.
If  $(-\infty,0)$ is regular, then its local time at 0 is the unique continuous, increasing, additive functional $L^*$ with $L_0^*=0$, a.s.,  such that the support of the 
measure $dL_t^*$ is the set $\overline{\{t:X_t=\underline{X}_t\}}$ and which is normalized by
\begin{equation}\label{norm1}
\e\left(\int_0^\infty e^{-t}\,dL_t^*\right)=1\,.
\end{equation}
Then the It\^o measure $n^*$ of the excursions away from 0 of the process $X-\underline{X}$ is characterized by the
{\it compensation formula}. More specifically, for any positive and predictable process $F$,
\begin{equation}\label{compensation}\e\left(\sum_{s\in G}F(s,\omega,\epsilon^s)\right)=
\e\left(\int_0^\infty dL_s^*\left(\int_E F(s,\omega,\epsilon)n^*(d\epsilon)\right)\right)\,,
\end{equation}
where $E$ is the set of excursions, $G$ is the set of left end points of the excursions, and $\epsilon^s$ is the excursion  which starts at $s\in G$.
We refer to \cite{bib:b96}, Chap.~IV, \cite {ky}, Chap.~6 and \cite{do} for more detailed definitions and some
constructions of $L^*$ and $n^*$.

When  $(-\infty,0)$ is not regular, the set $\{t:(X-\underline{X})_t=0\}$ is discrete and following
\cite{bib:b96} and \cite{ky}, we define the local time $L^*$ of $X-\underline{X}$ at 0 by
\begin{equation}\label{norm2}
L_t^*=\sum_{k=0}^{R_t}{\rm\bf e}^{(k)}\,,
\end{equation}
where for $t>0$, $R_t=\mbox{Card}\{s\in(0,t]:X_s=\underline{X}_s\}$, $R_0=0$ and ${\rm\bf e}^{(k)}$, $k=0,1,\dots$ is a sequence
of independent and exponentially distributed random variables with parameter
\begin{equation}\label{alpha}
\gamma=\left(1-\e(e^{-\tau^-_0})\right)^{-1}\,.
\end{equation}
In this case, the measure $n^*$ of the excursions away from 0  is proportional to
the distribution of the process $X$ under the law $\mathbb{P}$, killed at its first passage
time in the negative half line. More formally, let us define
$\epsilon^{0}=(X_t\ind_{\{t<\tau_0^-\}}+\infty\cdot\ind_{\{t\ge\tau_0^-\}})$,
then for any bounded Borel functional $K$ on $E$,
\begin{equation}\label{excdisc}
\int_EK(\epsilon)n^*(d\epsilon)=\gamma\,\e[K(\epsilon^{0})]\,.
\end{equation}
From definitions (\ref{norm2}), (\ref{excdisc}) and an application of the strong Markov property, we may check that the normalization 
(\ref{norm1}) and the compensation formula (\ref{compensation}) are still valid in this case.

In any case, $n^*$ is a Markovian measure whose semigroup is this of the killed L\'evy process when it enters in the negative half line.
More specifically, for $x>0$, let us denote by $\mathbb{Q}_{x}^*$ the law of the process
$(X_t\ind_{\{t<\tau_0^-\}}+\infty\cdot\ind_{\{t\ge\tau_0^-\}},\,t\ge0)$ under $\p_x$,
that is for $\Lambda \in \mathcal{F}_{t}$,
\begin{equation}\label{4524}
\mathbb{Q}_{x}^*(\Lambda ,t<\zeta )=\mathbb{P}_{x}(\Lambda ,\,t<\tau_{0}^-)\,.
\end{equation}
Then for all Borel positive functions $f$ and $g$ and for all $s,t>0$,
\begin{equation}\label{4527}
n^*(f(X_t)g(X_{s+t}),s+t<\zeta)=n^*(f(X_t)\e^{\mathbb{Q}^*}_{X_t}(g(X_s)),s<\zeta)\,,
\end{equation}
where $\e_x^{\mathbb{Q}^*}$ means the expectation under $\mathbb{Q}_x^*$. Recall that $\mathbb{Q}_{0}^*$ is well defined when $0$ 
is not regular for $(-\infty,0)$, and in this case, from (\ref{excdisc}), we have $\mathbb{Q}_{0}^*=\gamma^{-1}n^*$.  We define the probability 
measures $\mathbb{Q}_x$ in the same way as in (\ref{4524}) with respect to the dual process $(X,\p^*)$. Let us denote by $q_{t}^*(x,dy)$  
(resp. $q_{t}(x,dy)$)  the semigroup of the strong Markov  process $(X,\mathbb{Q}_x^*)$ (resp. $(X,\mathbb{Q}_x)$). 
Note  that from ($H_1$) and (\ref{4524}), the semigroups $q_t(x,dy)$ and $q_t^*(x,dy)$ are absolutely continuous. 
A slight extension of Lemma 2 in \cite{bib:u11} actually leads to the following result.
\begin{lemma}\label{7855} Under assumptions $(H_1)$ and $(H_2)$, for all $t>0$,  there are versions of the densities of the measures
$q_t(x,dy)$ and $q_t^*(x,dy)$  which are strictly positive and continuous on $(0,\infty)^2$. We denote by $q_t(x,y)$ and 
$q_t^*(x,y)$ these densities. Both $q_t$ and $q_t^*$ satisfy Chapman-Kolmogorov equations and the duality relation,
\begin{equation}\label{5687}
q_t^*(x,y)=q_t(y,x)\,,\;\;\;x,y>0,\,t>0\,.
\end{equation}
\end{lemma}
\begin{proof} It is obtained by following the proof of Lemma 2 in \cite{bib:u11} along the lines. Indeed, the latter result is proved under the 
additional assumptions that both half lines $(-\infty,0)$ and  $(0,\infty)$ are regular. But we can see that these properties are actually not needed, 
although regularity of $(-\infty,0)$ is argued at the beginning of this proof.
\end{proof}
\noindent Let us denote by $q_t^*(dx)$, $t>0$, the entrance law of $n^*$, that is for any positive Borel function $f$,
\begin{equation}\label{2682}
\int_{[0,\infty)}f(x)\,q_t^*(dx)=n^*\left(f(X_t),\,t<\zeta\right)\,.
\end{equation}
The local time at 0 of the reflected process at its supremum $\overline{X}-X=X^*-\underline{X}^*$ and the measure of its excursions away from 0
are defined in the same way as for $X-\underline{X}$. They are respectively denoted by $L$ and $n$.
Then the entrance law $q_t(dx)$ of $n$ is defined in the same way as $q_t^*(dx)$.
\begin{lemma}\label{7854} Under assumptions $(H_1)$ and $(H_2)$ the entrance laws $q_t(dx)$ and $q_t^*(dx)$ are absolutely continuous
on $[0,\infty)$ and there are versions of their densities which are strictly positive and continuous on $(0,\infty)$, for all $t>0$. We denote by 
$q_t(x)$ and  $q_t^*(x)$ these densities. Then both $q_t$ and $q_t^*$ satisfy Chapman-Kolmogorov equations: for $s,t>0$ and $y>0$,
\begin{equation}\label{4528}
q_{s+t}(y)=\int_0^\infty q_s(x)q_{t}(x,y)\,dx\;\;\;\mbox{and}\;\;\;q_{s+t}^*(y)=\int_0^\infty q_s^*(x)q_{t}^*(x,y)\,dx\,.
\end{equation}
\end{lemma}
\begin{proof} It suffices to prove the result for $q_t(dx)$.  It is proved in part 3.~of Lemma 1 in \cite{bib:l12}, that under assumption ($H_1$), 
the measure $q_t(dx)$ is absolutely continuous with respect to the Lebesgue measure on $[0,\infty)$.  Let $h_t$ be any version of its density 
and for all $s>0$ and $y>0$, define
\begin{equation}\label{9353}
q_{s,t}(y)=\int_0^\infty h_t(x)q_s(x,y)\,dx\,.
\end{equation}
We derive from $(H_1)$ and (\ref{4524}) (for the dual process) that $q_t(x,y)$ is uniformly bounded in $x,y\in(0,\infty)$. Moreover, from 
(\ref{2682}), (\ref{norm1}) and (\ref{compensation}), $\int_0^\infty h_t(x)\,dx=n(t<\zeta)<\infty$. Then from the Lebesgue dominated convergence theorem 
and Lemma \ref{7855},  relation (\ref{9353}) defines a continuous and strictly positive function on $(0,\infty)$. Moreover, from  
(\ref{4527}) we see that $q_{s,t}(x)$ is a density for $q_{t+s}(dx)$. Hence it only depends on $t+s$. Let us set $q_{s,t}(x)=q_{t+s}(x).$

Proceeding this way for all $s,t>0$, we define a family of strictly positive and continuous densities $q_t(x)$, $t>0$ of the entrance 
law of $n$ which satisfies the Chapman-Kolmogorov equations $q_{t+s}(y)=\int_0^\infty q_t(x)q_s(x,y)\,dx$, $x>0$, $s,t>0$.
\end{proof}
We end this section with the definition of the ladder processes. The ladder time processes $\tau$ and $\tau^*$, and the ladder height processes 
$H$ and $H^*$ are the following (possibly killed) subordinators:
\[\tau_t=\inf\{s:L_s>t\}\,,\;\;\tau^*_t=\inf\{s:L_s^*>t\}\,,\;\;H_t=X_{\tau_t}\,,\;\;H^*_t=-X_{\tau_t^*}\,,\;\;t\ge0\,,\]
where $\tau_t=H_t=+\infty$, for $t\ge\zeta(\tau)=\zeta(H)$ and $\tau_t^*=H_t^*=+\infty$, for $t\ge\zeta(\tau^*)=\zeta(H^*)$.
We denote by $\kappa$ and $\kappa^*$ the characteristic exponents  of the ladder processes $(\tau,H)$ and $(\tau^*,H^*)$. 
Recall that the drifts  ${\tt d}$ and ${\tt d}^*$  of the subordinators $\tau$ and $\tau^*$ satisfy
\begin{equation}\label{delta}
\int_0^t\ind_{\{X_s=\overline{X}_s\}}\,ds={\tt d}L_t\,,\;\;\;\int_0^t\ind_{\{X_s=\underline{X}_s\}}\,ds={\tt d}^*L_t^*
\end{equation}
and that ${\tt d}>0$ if and only if $(-\infty,0)$ is not regular.  In any case, we can check that ${\tt d}=\gamma^{-1}$, see \cite{bib:l12}.

\section{Main results}\label{mainres}

In all this section,  $(X,\p)$ is any L\'evy process satisfying assumptions ($H_1$) and ($H_2$). 
Then from Corollary 3 of \cite{bib:l12}, the law of the past supremum $\overline{X}_t$ on $[0,\infty)$ takes the following form,
\begin{equation}\label{5520}
\p(\overline{X}_t\in dx)=\int_0^t n(t-s<\zeta)q_s^*(x)\,ds\,dx+ {\tt d}q_t^*(x)\,dx+{\tt d}^*n(t<\zeta)\delta_{\{0\}}(dx)\,.
\end{equation}
Expression (\ref{5520}) shows that the law of $\overline{X}_t$ is absolutely continuous with respect to the Lebesgue measure on $(0,\infty)$. 
Moreover, this law has   an atom at 0 if and only if  $(0,\infty)$ is not regular. Then we will denote by $f_t(x)$ the following version of the density of 
$\p(\overline{X}_t\in dx)$ on $(0,\infty)$, 
\begin{equation}\label{5519}
f_t(x)=\int_0^t n(t-s<\zeta)q_s^*(x)\,ds+ {\tt d}q_t^*(x)\,,\;\;\;x>0\,.
\end{equation}
Note that there are instances where the law of $\overline{X}_t$ is absolutely continuous whereas assumption ($H_1$)  is not satisfied, 
see part 1 of Corollary 2 in \cite{bib:l12}. Expression $(\ref{5519})$ will be the starting point of our study. As the latter shows regularity properties of $f_t$,
such as continuity or asymptotic behaviour at 0, relate to those of $q_t^*$.  However, due to the 'bad' behaviour of the function $(t,x)\mapsto q_t^*(x)$, when 
$t$ and $x$ are small, some features of the first term on the right hand side of (\ref{5519}) cannot be directly derived from those of $q_t^*$. 
This study requires much sharper arguments which will be developed in the next section.\\

The next proposition extends Lemma 3 in \cite{bib:u11}. It describes the asymptotic behaviour at 0 of the functions $x\mapsto q_t^*(x,y)$ and 
$x\mapsto q_t^*(x)$. The second assertion is to be compared with 
Propositions 6 and 7 in \cite{dr} where similar results are obtained in the case where the law of $X$ is in the domain of attraction of a stable law.
\begin{proposition}\label{5012} For all $t>0$, 
\[\lim_{x\rightarrow0+}\frac{q_t^*(x,y)}{h^*(x)}=q_t^*(y)\,,\;\,\mbox{for all $y>0$ and}\;\;\;\lim_{x\rightarrow0+}\frac{q_t^*(x)}{h(x)}=\frac{p_t(0)}t\,,\]
where $h$ and $h^*$ are the renewal functions of the ladder height processes $H$ and $H^*$, that is $h(x)=\int_0^\infty\p(H_t\le x)\,dt$ and 
$h^*(x)=\int_0^\infty\p(H_t^*\le x)\,dt$, $x\ge0$.
\end{proposition}
\noindent In general, the function $h$ is finite, continuous, increasing and $h-h(0)$ is subadditive on $[0,\infty)$.
Moreover, $h(0) = 0$ if $(-\infty,0)$ is regular  and $h(0) = {\tt d}$ if not. 
This function is known explicitly only in the following cases: when $X$ has no positive jumps, $H$ is a pure drift. More specifically, 
given our normalisation of the local time $L$, one has $H_t=ct$, where $c=\Phi(1)$ and $\Phi$ is the Laplace exponent of the subordinator
$T_x=\inf\{t:X_t>x\}$, $x\ge0$,  so that $h(x)=c^{-1}x$. When $X$ is a stable process with index $\alpha\in(0,2]$ and positivity coefficient 
$\p(X_1>0)=\rho$, then $H$ is a stable subordinator with index $\alpha\rho$, and $h(x)=\e(H_1^{-\alpha\rho})x^{\alpha\rho}$. 
Finally, when the characteristic exponent of $X$ is of the form $\Psi(\xi)=\psi(\xi^2)$ for a complete  Bernstein function $\psi$, then $h(x)$ is a Bernstein 
function and its integral representation in terms of $\psi(\xi)$ was given in Proposition 4.5 in \cite{bib:kmr12}.\\

Recall from (1.8)  and (3.3) in \cite{si}, see also parts 2 and 3 of Lemma 1 in \cite{bib:l12}, that 
the renewal function $h$ of the upward ladder process $H$ is everywhere differentiable and that its derivative is given by
\begin{equation}\label{9421}
h'(x)=\int_0^\infty q_s^*(x)\,ds\,,\;\;\;\mbox{for all $x>0$.}
\end{equation}
Moreover Lemma \ref{7854} ensures that
\[h'(x)>0\,,\;\;\;\mbox{for all $x>0$.}\]
\noindent Knowing that $x\mapsto q_t^*(x)$ is continuous on $(0,\infty)$ and considering the representation (\ref{5519}), it is natural to ask about
continuity of $f_t$. 
\begin{proposition}\label{thm:continuity}
   The following conditions are equivalent$:$
   \begin{enumerate}
      \item[(1)] $x\to h'(x)$ is continuous at $x_0>0$,
      \item[(2)] $x\to f_t(x)$ is continuous at $x_0>0$ for every $t>0$,
      \item[(3)] $x\to f_t(x)$ is continuous at $x_0>0$ for some $t>0$.
   \end{enumerate}
\end{proposition}
\noindent The function $h'$ is known to be continuous on $(0,\infty)$ in many instances. We have already seen that it is the case when $X$ is a stable 
process. It is also continuous when the process has no positive jumps, but more generally if the ascending ladder height process $H$ has a positive 
drift, then $h'$ is continuous and bounded, see Theorem 19, Section VI.4 in \cite{bib:b96}.
Continuity of $h'$ can be also deduced from Proposition 4.5 in \cite{bib:kmr12} for a wide class of subordinated Brownian motions.
Actually, this function is not always continuous, see for instance Lemma 2.4 in \cite{kkr}, where it is proved that if $X$ has no positive jumps, bounded
variations and a L\'evy measure which admits atoms, then $h'$ is not continuous. \\

\noindent Then a subsequent question concerns the asymptotic behaviour of $f_t$ at 0, for which we have the following result.

\begin{theorem}\label{4883} The  density of the law of the past supremum of $(X,\p)$ fulfills the following asymptotic behaviour,
\formula{
\lim_{x\to 0^{+}}\frac{f_t(x)}{h'(x)} =  n(t<\zeta)\,,
}
uniformly on $[t_0,\infty)$ for every fixed $t_0>0$.
\end{theorem}

We now state two results regarding the asymptotic behaviour of $f_t(x)$, when $t$ tends to infinity. First recall the following equivalent forms
of Spitzer's condition. Let $\rho\in(0,1)$, and denote by $R_\rho(0)$ (resp. $R_{-\rho}(\infty)$) the set of regularly varying functions at 0+ (resp. at $+\infty$)
with index $\rho$ (resp. $-\rho$), then 
\begin{equation}\label{4512}
\lim_{t\rightarrow\infty}\p(X_t\ge0)=\rho\;\;\;\Leftrightarrow\;\;\;\alpha\mapsto\kappa(\alpha,0)\in R_\rho(0)\;\;\;\Leftrightarrow\;\;\;
t\mapsto n(t<\zeta)\in R_{-\rho}(\infty)\,.
\end{equation}
The first equivalence can be found in Theorem 14, Section VI.3 in \cite{bib:b96}, see also the discussion after this theorem. The second equivalence follows from the 
discussion after Theorem 6, Section III.3 in \cite{bib:b96} and the identity $n(t<\zeta)=\pi(t,\infty)+a$, where $\pi$ is the L\'evy measure of $\tau$ and $a$ its killing rate.  
Then Theorem \ref{thm:t:asymptotic} provides a uniform limit in $x$ on compact sets, under assumption (\ref{4512}). 
This result complements, and in some cases generalizes, the result of \cite{bib:gn86}, where the same study was performed for the 
distribution function  $\p(\overline{X}_t\le x)$. 
\begin{theorem}
  \label{thm:t:asymptotic} Assume that  $(\ref{4512})$ is satisfied, then 
  \formula{
     \lim_{t\to\infty}\frac{f_t(x)}{n(t<\zeta)} = h'(x)\,,
  }
  uniformly in $x$ on every compact subset of $(0,\infty)$.
\end{theorem}

\noindent The next theorem provides some general estimates for $f_t(x)$, when $t\geq t_0$ and $x\leq x_0$, for any given $t_0,x_0>0$. 
These estimates are sharp when (\ref{4512}) is satisfied. 

\begin{theorem}\label{thm:t:estimates}
  For fixed $x_0,t_0>0$ there exist positive constants $c_1$ and $c_2$ such that
  \formula{
     c_1\,n(t<\zeta) \leq \frac{f_t(x)}{h'(x)} \leq c_2\,\frac{1}{t}\int_0^t n(s<\zeta)\,ds \/,\quad x\leq x_0\/,t\geq t_0\/.
  }
  If additionally $(\ref{4512})$ is satisfied, then there exists $c_3>0$ such that
  \formula{
     c_1\,h'(x)\,n(t<\zeta) \leq  f_t(x) \leq c_3\,h'(x)\,n(t<\zeta)\/,\quad x\leq x_0\/,t\geq t_0\/.
  }
\end{theorem}

Now we derive from Proposition \ref{5012} the asymptotics of the densities of the L\'evy process $(X,\p)$ conditioned to stay positive and this of its meander. 
L\'evy processes conditioned to stay positive will also be involved in the proofs of Section \ref{proofs}. Let us briefly recall their definition which may 
be found in more details in  \cite{bib:cd05} and \cite{bib:cd10}. The law of the L\'evy process  $(X,\p)$ conditioned to stay positive is a Doob $h$-transform 
of the killed process $(X,\mathbb{Q}_{x}^*)$ defined in (\ref{4524}).  
It is obtained  from the renewal function $h^*$ of the downward ladder height process $H^*$ which is excessive for $(X,\mathbb{Q}_{x}^*)$ and 
invariant if and only if $\limsup_{t\rightarrow\infty}X_t=+\infty$, a.s. The  conditioned process  is currently denoted by $(X,\p_x^\uparrow)$ and 
formally defined by
\begin{equation}\label{7212}
\p_x^\uparrow(\Lambda,t<\zeta)=\frac1{h^*(x)}\e^{\mathbb{Q}^*}_x(h^*(X_t)\ind_{\{\Lambda,t<\zeta\}})\,,\;\;\;x>0\,,\;\;\;\Lambda\in\mathcal{F}_t\,.
\end{equation}
We also recall from Theorem 2 in \cite{bib:cd05} that the family of measures $(\p_x^\uparrow)$ converges as $x\downarrow0$, toward a probability 
measure $\p^\uparrow$ which is related to $n^*$ by the following expression:
\begin{equation} \label{6253}
\p^\uparrow(\Lambda,t<\zeta)=n^*(h^*(X_t)\ind_{\{\Lambda,t<\zeta\}})\,.
\end{equation}
This convergence holds weakly on the Skohorod's space when $(0,\infty)$ is regular and in a more specific sense when this half line is not regular.
In any case, we derive from (\ref{6253}) that  the density of the law $\p^\uparrow(X_t\in dx)$, for $t>0$ is related to the entrance law $q_t^*$ as follows:
\begin{equation}\label{5317}
p^\uparrow_t(x)=h^*(x)q_t^*(x)\,.
\end{equation}

The meander with length $t>0$, is a process with the law of $(X_s,\,0\le s\le t)$ under the conditional distribution $\p(\,\cdot\,|\,\underline{X}_t\ge0)$. 
This conditioning only makes sense when $(-\infty,0)$ is not regular. When $(0,\infty)$ is regular, it corresponds to the law of  $(X_s,\,0\le s\le t)$ under the limiting 
probability  measure
\[M^{(t)}:=\lim_{x\downarrow0}\frac1{h^*(x)}\p_x(\,\cdot\,|\,\underline{X}_t\ge0)\,.\]
A general definition can be found in \cite{bib:cd10},  see Section 4 and relation  (4.5) therein. It implies in particular that on $\mathcal{F}_t$, the law 
$M^{(t)}$ of the  meander of length $t$ is absolutely continuous with respect to the process $(X,\p^\uparrow)$,  with density $(h^*(X_t))^{-1}$. 
As a consequence, the density  of the distribution $M^{(t)}(X_t\in dx)$ of the meander with length $t$ at time $t$, which we denote by $m_t(x)$, is given by:
\[m_t(x)=n^*(t<\zeta)^{-1}q^*_t(x)\,.\]
This relation together with  (\ref{5317}) lead to the following straightforward consequence of Proposition \ref{5012}.
\begin{corollary} The density $m_t(x)$ of the law of the meander with length $t$, at time $t$ and the density $p_t^\uparrow(x)$ of the entrance  law of the 
L\'evy process conditioned to stay positive are continuous and strictly positive on $(0,\infty)$. Moreover they have the following asymptotic behaviour at $0$:
\[m_t(x)\sim \frac{p_t(0)}{t\,n^*(t<\zeta)}h(x)\;\;\;\mbox{and}\;\;\;p_t^\uparrow(x)\sim \frac{p_t(0)}th(x)h^*(x)\,,\;\;\;
\mbox{as}\;\;\;x\rightarrow 0\,.\]
\end{corollary}

\section{Proofs}\label{proofs}
\noindent Before proceeding to the proofs of the theorems, we need a couple of additional  preliminary results. We first extend 
Corollary 1 of \cite{bib:cd05} to the case where $(0,\infty)$ is not regular. Recall from (\ref{7212}) the definition of L\'evy processes conditioned to stay positive.   
\begin{proposition}\label{7632} Assume that $(X,\p)$ is not a compound Poisson process and that  $(|X|,\p)$ is not a subordinator.
Then for all bounded and continuous function $f$ and for all $t>0$,
\[\lim_{x\rightarrow0}\e^{\uparrow}_x(h^*(X_t)^{-1}f(X_t))=n^*(f(X_t),\,t<\zeta)\,.\]
\end{proposition}
\begin{proof} When $(0,\infty)$ 
is  regular for  $(X,\p_x)$, then the result is Corollary 1 of \cite{bib:cd05} whose  proof is given in \cite{bib:cd08}.

Let us assume that $(0,\infty)$ is not regular for $(X,\p_x)$.
Then  from the second part of Theorem 2 of \cite{bib:cd05} and relation (3.2) in this article, we still have for all $t>0$,
\begin{equation}\label{2573}
\lim_{x\rightarrow0}\e^{\uparrow}_x(f(X_t))=n^*(h^*(X_t)f(X_t),\,t<\zeta)\,.
\end{equation}
(Note that the constant $k$ in (3.2) of \cite{bib:cd05} is equal to 1, according to the normalisation of the local time that is recalled in (\ref{norm1}).)
However, since $h(0)=0$, the function $x\mapsto h(x)^{-1}f(x)$ is not necessarily bounded, so we cannot replace $f$ by this function in (\ref{2573}) in 
order to get our result. But from the weak convergence stated in (\ref{2573}), we may derive that for all fixed $\delta>0$ and $t>0$,
\begin{equation}\label{1236}
\lim_{x\rightarrow0}\e^{\uparrow}_x\left(h^*(X_t)^{-1}f(X_t)\ind_{\{X_t>\delta\}}\right)=n^*(f(X_t),X_t>\delta,\,t<\zeta)\,.
\end{equation}
In particular, with $f\equiv1$, we obtain from definitions (\ref{4524}) and (\ref{7212}), that
\begin{equation}\label{1237}
\lim_{x\rightarrow0}h^*(x)^{-1}\p_x(X_t>\delta,\tau_0^->t)=n^*(X_t>\delta,\,t<\zeta)\,,
\end{equation}
so that if we can prove  
\begin{equation}\label{4732}
\lim_{x\rightarrow0}h^*(x)^{-1}\p_x(\tau_0^->t)=n^*(t<\zeta)\,,
\end{equation}
then taking the difference between (\ref{1237}) and (\ref{4732}), we will derive that,
\[\lim_{x\rightarrow0}\e_x^\uparrow(h^*(X_t)^{-1}\ind_{\{X_t\le \delta\}})=n^*(X_t\le \delta,\,t<\zeta)\,.\]
Since $n^*(X_t=0,\,t<\zeta)=0$ and $f$ is uniformly bounded by $K$, we will obtain
\[\lim_{\delta\rightarrow0}n^*(f(X_t)\ind_{\{X_t\le \delta\}})\le\lim_{\delta\rightarrow0}Kn^*(X_t\le\delta,\,t<\zeta)=0\,,\]
and the result will follow.

Then let us prove (\ref{4732}). This point differs from the proof given in \cite{bib:cd08} in the regular case. 
First recall formula (1) in  \cite{bib:cd08}:
\begin{equation}\label{4731}\p_x(\tau_0^->e/\varepsilon)=\e\left(\int_0^\infty  
e^{-\varepsilon s}\ind_{\{\underline{X}_s\ge-x\}}\,dL_s^*\right)[{\tt d}^*\varepsilon+n^*(e/\varepsilon<\zeta)]\,,
\end{equation}
which can be derived from the compensation formula (\ref{compensation}) and (\ref{delta}). 
Set $h^{(\varepsilon)}(x):=\e\left(\int_0^\infty  e^{-\varepsilon s}\ind_{\{\underline{X}_s\ge-x\}}\,dL_s^*\right)$ and recall that
$h^*(x)=\e\left(\int_0^\infty  \ind_{\{\underline{X}_s\ge-x\}}\,dL_s^*\right)$. Then we will first show that for all $\varepsilon>0$,
\begin{equation}\label{8255}
 h^{(\varepsilon)}(x)\sim h^*(x)\,,\;\;\;\mbox{as $x\rightarrow0$.}
\end{equation}
First note that for all $\varepsilon>0$, $h^{(\varepsilon)}(x)\le h^*(x)$. Then, for the lower bound, we can write for all $u>0$, 
$h^{(\varepsilon)}(x)\ge e^{-\varepsilon u}\e\left(\int_0^u  \ind_{\{\underline{X}_s\ge-x\}}\,dL_s^*\right)$, so that 
\begin{eqnarray}
h^*(x)&=&\e\left(\int_0^u  \ind_{\{\underline{X}_s\ge-x\}}\,dL_s^*\right)+\e\left(\int_u^\infty  \ind_{\{\underline{X}_s\ge-x\}}\,dL_s^*\right)
\nonumber\\
&\le& e^{\varepsilon u}h^{(\varepsilon)}(x)+\e\left(\int_u^\infty  \ind_{\{\underline{X}_s\ge-x\}}\,dL_s^*\right)\label{4673}\,.
\end{eqnarray}
Then applying the Markov property at time $u$, we obtain that 
$\e\left(\int_u^\infty  \ind_{\{\underline{X}_s\ge-x\}}\,dL_s^*\right)\le\p_x(\tau_0^-\ge u)h^*(x)$. Plunging this in (\ref{4673}), we get
\[h^*(x)\le\frac{e^{\varepsilon u}}{1-\p_x(\tau_0^-\ge u)}h^{(\varepsilon)}(x)\,.\]
Observe that since $(-\infty,0)$ is regular, for all $u>0$, $\lim_{x\rightarrow0}\p_x(\tau_0^-\ge u)=0$.
Let $\delta>1$, then from the above inequality, for $u$ sufficiently small, we can find $x_0>0$ such that for all $x\le x_0$, $h^*(x)\le \delta h^{(\varepsilon)}(x)$.
So we have proved (\ref{8255}). Then let us rewrite (\ref{4731}) as follows:
\[\int_0^\infty e^{-\varepsilon s}\p_x(\tau_0^->s)\,ds=h^{(\varepsilon)}(x)[{\tt d}^*+\int_0^\infty e^{-\varepsilon s}n^*(s<\zeta)\,ds]\,.\]
From (\ref{8255}), we obtain that for all $\varepsilon>0$,
\[\lim_{x\rightarrow0}\int_0^\infty e^{-\varepsilon s}\frac{\p_x(\tau_0^->s)}{h^*(x)}\,ds={\tt d}^*+\int_0^\infty e^{-\varepsilon s}n^*(s<\zeta)\,ds\,,\]
which means that the measure with density $s\mapsto \p_x(\tau_0^->s)/h^*(x)$ converges weakly toward the measure ${\tt d}^*\delta_0(ds)+n^*(s<\zeta)\,ds$,
as $x$ tends to 0. 

Then from this fact, we can derive (\ref{4732}) as it is done in the proof of Corollary 1 in \cite{bib:cd08}. Let $c\in(0,t)$, then 
\begin{eqnarray*}
\lim_{x\rightarrow0}h(x)^{-1}\p_x(\tau_0^->t)&\ge&c^{-1}\lim_{x\rightarrow0}h(x)^{-1}\int_t^{t+c}\p_x(\tau_0^->s)\,ds\\
&=&c^{-1}\int_t^{t+c}n^*(\zeta>s)\,ds\ge n^*(\zeta>t+c)\\
\lim_{x\rightarrow0}h(x)^{-1}\p_x(\tau_0^->t)&\le&c^{-1}\lim_{x\rightarrow0}h(x)^{-1}\int_{t-c}^{t}\p_x(\tau_0^->s)\,ds\\
&=&c^{-1}\int_{t-c}^{t}n^*(\zeta>s)\,ds\le n^*(\zeta>t-c)\,,\\
\end{eqnarray*}
and the result follows, since $c$ can be chosen arbitrarily small.
\end{proof}
\noindent Let $\p^{*\uparrow}_x$, $x\ge0$ be the law of the dual L\'evy process $(X,\p^{*}_x)$  conditioned to stay positive. Then Proposition \ref{7632}
is interpreted for the dual process as follows:
\begin{equation}\label{2632}
\lim_{x\rightarrow0}\e^{*\uparrow}_x(h(X_t)^{-1}f(X_t))=n(f(X_t),\,t<\zeta)\,,
\end{equation}
for all bounded and continuous function $f$ and for all $t>0$. It is actually under the latter form that Proposition \ref{7632} will be used in the proof of  
Proposition \ref{5012} below.\\

In the next results, we will use some properties of the bridge of $(X,\p)$. Let us now briefly recall its definition.  We refer to Section VIII.3 of \cite{bib:b96} 
for a more complete account on the subject. Assume that $(H_1)$ and $(H_2)$ are satisfied, then the law $\p^{t}_{x,y}$ of the bridge from $x\in\mathbb{R}$ 
to $y\in\mathbb{R}$, with length $t>0$ of the L\'evy process $(X,\p)$ is a regular version of the conditional law of $(X_s,\,0\le s\le t)$ given $X_t=y$, under $\p_x$. 
It satisfies  $\p^{t}_{x,y}(X_0=x,X_t=y)=1$ and for all $s<t$, this law is absolutely continuous with respect to $\p_x$ on ${\mathcal F}_s$, with density 
$p_{t-s}(X_s-x)/p_t(y-x)$, i.e.
\begin{equation}\label{4573}
\p_{x,y}^t(\Lambda)=\e\left(\ind_{\Lambda}\frac{p_{t-s}(X_s-x)}{p_t(y-x)}\right)\,,\;\;\;\mbox{for all $\Lambda\in{\mathcal F}_s$}\,.
\end{equation}

In the next proposition, we give the law of the time at which the bridge $(X,\p_{0,y}^t)$, reaches its supremum over $[0,t]$. 
Note that since this time occurs only once, a.s. for the process $(X,\p)$, then  
the same property holds for $(X,\p_{0,y}^t)$. This fact can easily be derived from (\ref{4573}). 
Let us denote by $g_t$ this time, i.e.
\[g_t=\sup\{s\le t:X_s=\overline{X}_s\;\;\mbox{or}\;\;X_{s-}=\overline{X}_s\}\,.\]
\begin{proposition}\label{8432} Assume that $(H_1)$ and $(H_2)$ are satisfied. Then for all $y\in\mathbb{R}$ the law of the time of the supremum of the 
bridge $(X,\p_{0,y}^t)$ is absolutely continuous on $[0,t]$ and its density  is given by:
\[\frac{\p_{0,y}^t(g_t\in ds)}{ds}=p_t(y)^{-1}\int_0^\infty q^*_s(x)q_{t-s}(x+y)\,dx\,,\;\;\;s\in[0,t]\,.\]
\end{proposition}
\begin{proof} The result is a direct consequence of Theorem 3 in \cite{bib:l12} which asserts that
\begin{eqnarray*}
&\p(g_t\in ds,\,\overline{X}_t\in dx,\,X_t-\overline{X}_t\in dy)=\\
&q_t^*(x)q_{t-s}(y)\ind_{[0,t]}(s)\,ds\,dx\,dy+{\tt d}\delta_{\{t\}}(ds)q_t^*(x)\delta_{\{0\}}(dy)\,dx+{\tt d}^*\delta_{\{0\}}(ds)\delta_{\{0\}}(dx)q_t(y)\,dy\,.
\end{eqnarray*}
\end{proof}
\noindent For $y=0$, the time $g_t$ of the supremum of the bridge $(X,\p_{0,y}^t)$ is uniformly distributed over $[0,t]$,
see \cite{kn}. Then as a consequence of this result and Proposition \ref{8432}, we obtain  the following equality:
\begin{equation}\label{7764}\mbox{for all $s\in(0,t)$, $\displaystyle \int_0^\infty q^*_{t-s}(x)q_{s}(x)\,dx=\frac{p_t(0)}t$.}
\end{equation}

\vspace*{.2in}
\noindent {\it Proof of Proposition $\ref{5012}$}. When both half lines $(-\infty,0)$ and $(0,\infty)$ are regular, the result  follows directly from 
Lemma 3 of \cite{bib:u11}. This lemma actually concerns the transition densities $p_t^{*\uparrow}(x,y)$ of the process $(X,\p_x^{*\uparrow})$, but it 
is easily interpreted in terms of the transition densities $q_t^*(x,y)$ and the entrance law $q_t^*(x)$, thanks to relations (\ref{7212}) and (\ref{5317}).
Actually Lemma 3 of \cite{bib:u11} yields
\begin{equation}\label{7364}
\lim_{y\rightarrow0}\frac{q_t^*(y)}{h(y)}=\int_0^\infty q^*_{t-s}(x)q_{s}(x)\,dx\,,
\end{equation}
and we conclude to the second assertion from identity (\ref{7764}).

Now let us consider the case where one of the half lines is not regular. Note that the main argument in the proof of Lemma 3 in \cite{bib:u11} is the fact that for 
all $t>0$, $\lim_{x\rightarrow0}\e_x^{*\uparrow}(h(X_t)^{-1})=\e^{*\uparrow}(h(X_t)^{-1})=n^*(t<\zeta)<\infty$, which we have proved in Proposition 
\ref{7632}, in the general case. (Here we actually use the result for the dual process, see (\ref{2632})). Thanks to this result, we can follow the proof of Lemma 3 of 
\cite{bib:u11} along the lines in order to check that it is still valid, when one of the half lines is not regular. Then we conclude as above.$\;\;\Box$\\

\noindent A key point in the proof of our main result is the following proposition regarding integrability properties of $t\to {p_t(0)}/{t}$, 
both at zero and at infinity.

\begin{proposition}
      If there exists $t_0>0$ such that $x\to p_{t_0}(x)$ is bounded, then 
   \formula[ptt:infinity]{
      \int^\infty \frac{p_t(0)}{t}dt<\infty\/.
    }
   Moreover, if $t\to p_{t}(x)$ is bounded for every $t>0$ $($that is $(\mbox{H}_1)$ holds$)$ then 
   \formula[ptt:zero]{
      \int_{0^{+}} \frac{p_t(0)}{t}dt=\infty\/.
   }
\end{proposition}
\begin{proof}
Since boundedness of $x\to p_{t_0}(x)$ implies that $p_{t_0}\in L^2(\mathbb{R})$, its Fourier transform is also in $L^2(\mathbb{R})$ which means that 
$e^{-2t_0\textrm{Re}(\Psi(\cdot))}\in L^1(\mathbb{R})$. 
On the one hand it implies integrability of the characteristic function of $X$ for $t\geq 2t_0$ and, by the Riemann-Lebesgue lemma, continuity of $p_t$ for
$t\geq 2t_0$. On the other hand, applying inverse Fourier transform together with Fubini-Tonelli theorem, we can write
\formula{
  \int_{2t_0}^\infty \frac{p_t(0)}{t}dt \leq \frac{1}{2\pi}\int_{2t_0}^\infty \frac{1}{t}\int_{\mathbb{R}}|e^{-t\Psi(\xi)}|\,d\xi\,dt
= \frac{1}{2\pi} \left(\int_0^1+\int_1^\infty\right)  \int_{2t_0}^\infty\frac{1}{t}e^{-2t\textrm{Re}(\Psi(\xi))}dt\,d\xi\,.
}
By integrability of the  characteristic function of $X_{2t_0}$ and the fact that $\textrm{Re}\Psi(1)>0$, we obtain 
\formula{
\int_1^\infty \int_{2t_0}^\infty \frac{1}{t}e^{-2t\textrm{Re}(\Psi(\xi))}dt\,d\xi \leq \int_1^\infty 
e^{-2t_0\textrm{Re}(\Psi(\xi))}\,d\xi\cdot \int_{2t_0}^\infty \frac{1}{t}e ^{-t\textrm{Re}(\Psi(1))}dt<\infty\/,
}
hence it is enough to show the finiteness of the integral over $(0,1)$. Recall that $2\textrm{Re}\Psi(\xi)$ is the L\'evy-Khintchin exponent of the symmetrization of $X$. 
Thus, by L\'evy-Khintchin formula, there exists a constant $c>0$ such that $\Psi(\xi)\geq c\xi^2$ whenever $\xi \in (0,1)$. Moreover, we have
\formula{
   \int_{2t_0}^\infty \frac{1}{t}e^{-ct\xi^2}dt\approx -\ln \xi\/,\quad \xi \to 0\/.
}
It finally gives
\formula{
\int_0^1 \int_{2t_0}^\infty \frac{1}{t}e^{-2t\textrm{Re}(\Psi(\xi))}dt\,d\xi\leq \int_0^1 \int_{2t_0}^\infty \frac{1}{t}e^{-ct\xi^2}dt\,d\xi<\infty\/,
}
which ends the proof of (\ref{ptt:infinity}). 

To deal with (\ref{ptt:zero}) recall that  $(H_1)$ implies that the function $t\to p_t(0)$ is completely monotone 
(see \cite{bib:ssv10} p.118), so in particular it is decreasing. It entails
\formula{
   \int_0^{t_1}\frac{p_t(0)}{t}\,dt\geq p_{t_1}(0)\int_0^{t_1}\frac{dt}{t} =  \infty.
}
This ends the proof.
\end{proof}

\noindent We are now ready to proceed to the proofs of our main results.\\

\noindent{\it Proof of Theorem $\ref{4883}$}: Let us first note that
\begin{equation}\label{9255}
\lim_{x\to0^{+}}\frac{h(x)}{h'(x)}=0\,.
\end{equation}
Indeed, from (\ref{9421}), Proposition \ref{5012} and the Fatou Lemma, we have
\[\liminf_{x\to0^{+}}\frac{h'(x)}{h(x)}=\liminf_{x\to0^{+}}\frac{1}{h(x)}\int_0^\infty q^*_s(x)\,ds\ge
\int_0^\infty \liminf_{x\to0^{+}} \frac{q^*_s(x)}{h(x)}\,ds=\int_0^\infty\frac{p_s(0)}s\,ds\/,\]
which is infinite from (\ref{ptt:zero}).

Secondly, we note that under $(H_1)$, for every $t>0$ we have
\formula[eq:qt:upper:general]{
   q_t^*(x,y)\leq p_t(x-y) \leq \frac{1}{2\pi}\int_{\mathbb{R}}e^{-2t\textrm{Re}\Psi(\xi)}\,d\xi = p_t^S(0)\/,
}
where the first inequality follows from (\ref{4524}) and where $p_t^S = p_t*p_t$ is the density of the semi-group of the symmetrization of $X$. 
By the upper-bounds given in Theorem 3.1 in \cite{bib:kmr12} we have
\formula[sup:kmr:estimate]{
  \int_0^\infty q_t(x,y)dy = {\p(\overline{X}_{t}\le x)}\leq \frac{e}{e-1}\kappa(1/t,0)h(x)\/.
}
Note that this bound is true for every L\'evy process and that an analogous result holds for the reflected process $\overline{X}-X$. 
Then, applying the Chapman-Kolmogorov equation and  unsing the inequalities (\ref{eq:qt:upper:general}) and (\ref{sup:kmr:estimate}), we obtain
\formula{
q_{3t}^*(x,y) &= \int_0^\infty \int_0^\infty q_t^*(x,z)q_t^*(z,w)q_t^*(w,y)dzdw\\
  &\leq p_t^S(0)\int_0^\infty q_t^*(x,z)dz\int_0^\infty q_t^*(w,y)\,dw\\
  &= p_t^S(0) \int_0^\infty q_t^*(x,z)dz \int_0^\infty q_t(y,w)\,dw\\
  &\leq  \left(\frac{e}{e-1}\right)^2p_t^S(0)h^*(x)h(y)\kappa(1/t,0)\kappa^*(1/t,0)\/.
}
This inequality together with the Wiener-Hopf factorization  $\kappa(1/t,0)\kappa^*(1/t,0)=1/t$ yields 
\formula[qt:upperbound]{
  \frac{q_{3t}^*(x,y)}{h^*(x)h(y)}\leq  \left(\frac{e}{e-1}\right)^2\frac{p_t^S(0)}{t}\,.
}
Taking the limit when $x\to 0$ and using Proposition \ref{5012} we can finally write 
\formula[qth:integrability]{
   \frac{q_t^*(y)}{h(y)}\leq 3 \left(\frac{e}{e-1}\right)^2\frac{p_{t/3}^S(0)}{t}\/,\quad y>0, t>0\/.
}
Similarly, applying Chapman-Kolmogorov equation (\ref{4528}), we can write for $\delta\in (0,s)$
\formula{
  \frac{q_s^*(x)}{h(x)} &= \int_0^\infty \int_0^\infty q_{s-\delta}^*(z)q_{\delta/2}^*(z,w)\frac{q_{\delta/2}^*(w,x)}{h(x)}dz\,dw\\
  &\leq p_{\delta/2}^S(0)\int_0^\infty q_{s-\delta}^*(z)dz\cdot \int_0^\infty \frac{q_{\delta/2}(x,w)}{h(x)}dw\/.
}
Consequently, using (\ref{sup:kmr:estimate}) together with the fact that $\int_0^\infty q_{s-\delta}^*(z)dz=n^*(s-\delta<\zeta)$, we get 
\formula[eq:gsh:estimate2]{
\frac{q_s^*(x)}{h(x)} \leq c_{\delta}n^*(s-\delta<\zeta)\/,\quad x>0\/,
}
where
\formula{
  c_\delta = \frac{e}{e-1} p_{\delta/2}^S(0)\kappa(2/\delta,0)\,.
}
Since $t\to n(t<\zeta)$ is a continuous, nonnegative and decreasing function, it is uniformly continuous on $[t_0/2,\infty)$. 
For every $\varepsilon>0$ we can choose $0<\delta<t_0/2$ such that
\formula{
   n(t-\delta<\zeta)-n(t<\zeta)\leq \varepsilon\/,\quad t\geq t_0\/.
} 
Then, using (\ref{5519}), monotonicity of $t\to n(t<\zeta)$, (\ref{9421}) and (\ref{eq:gsh:estimate2}), we can write for every $t\geq t_0$ that
\formula{
   \frac{f_t(x)}{h'(x)} &= \int_0^\delta n(t-s<\zeta)\frac{q_s^*(x)}{h'(x)}ds + \frac{h(x)}{h'(x)}\int_\delta^t n(t-s<\zeta)\frac{q_s^*(x)}{h(x)}ds + {\tt d}\frac{h(x)}{h'(x)}\frac{q_t^*(x)}{h(x)}\\
   & \leq n(t-\delta<\zeta) + \frac{h(x)}{h'(x)}\left[c_\delta\int_\delta^tn(t-s<\zeta)n^*(s-\delta<\zeta)ds+ {\tt d}c_{t_0/2}n^*(t-t_0/2<\zeta)\right]\/,
   }
where the last term was estimated using (\ref{eq:gsh:estimate2}) with $s:=t$ and $\delta := t_0/2$. The following simple inequality
\formula{
\int_{0}^{t-\delta} n(t-\delta-s<\zeta)n^*(s<\zeta)\,ds \leq \int_{[0,\infty)}f_{t-\delta}(x)dx \le 1\/,
}
which is a consequence of integration of the formula (\ref{5519}) with respect to $x$, the choice of $\delta$ and the monotonicity of $n(\cdot<\zeta)$ give
   \formula{
   \frac{f_t(x)}{h'(x)} & \leq n(t<\zeta)+\varepsilon + \frac{h(x)}{h'(x)}\left[c_\delta+c_{t_0/2}{\tt d}n^*(t_0/2<\zeta)\right]\/,
}
for every $t>t_0$ and $x>0$. 
Consequently, using (\ref{9255}) and the fact that $\varepsilon$ was arbitrary, we have
\formula{
   \limsup_{x\to 0^+} \frac{f_t(x)}{h'(x)}\leq n(t<\zeta)\,,
}
uniformly on $[t_0,\infty)$. For the lower bound, note that monotonicity of $t\to n(t<\zeta)$ and (\ref{5519}) give
\formula{
  \frac{f_t(x)}{h'(x)}\geq n(t<\zeta)\int_0^t\frac{q_s^*(x)}{h'(x)}ds\geq n(t<\zeta)-n(t<\zeta)\frac{h(x)}{h'(x)}\int_t^\infty \frac{q_s^*(x)}{h(x)}ds\/.
}
Since from (\ref{qth:integrability}) we have
\formula{
  n(t<\zeta)\int_t^\infty \frac{q_s^*(x)}{h(x)}ds\leq \left(\frac{e}{e-1}\right)^2 n(t_0<\zeta)\int_{3t_0}^\infty \frac{p_t^S(0)}{t}dt \/,\quad t\geq t_0\/.
}
Note that boundedness of $x\to p_t(x)$ implies boundedness of $p_t^S$ (since $p_t^S$ is a convolution of a function from $L^1(\mathbb{R})$ and a bounded function) 
and consequently, by (\ref{ptt:infinity}) and (\ref{9255}) we finally obtain
\formula{
   \liminf_{x\to 0^+}\frac{f_t(x)}{h'(x)}\geq n(t<\zeta)\/,\quad \textrm{uniformly for }t\geq t_0\/.
}
This ends the proof.
$\Box$\\

\noindent{\it Proof of Theorem $\ref{thm:t:asymptotic}$}.
 Let $A$ be any compact subset of $(0,\infty)$. Since $t\mapsto n(t<\zeta)$ is regularly varying at infinity, we have
   \formula[eq:nt:compare]{
      \frac1t\int_0^t n(s<\zeta)ds \approx n(t<\zeta)\/,\quad t\to \infty. 
   }
Here $ f(t) \approx g(t),t\to\infty$ means that there exists constant $c>1$ such that $c^{-1} g(t)\leq f(t)\leq c g(t)$ for large $t$.
Then let us split formula (\ref{5519}) into two parts by writing $f_t^{1}(x)$ for the integral component and $f_t^2(x):={\tt d}q_t^*(x)$.
  Thus, for every fixed $\delta\in(0,1)$, by monotonicity of $n(\cdot<\zeta)$ and (\ref{qth:integrability}) we have
  \formula{
     f_t^1(x) &= \left(\int_0^{(1-\delta) t}+\int_{(1-\delta) t}^t \right)n(s<\zeta)q_{t-s}^*(x)\,ds\\
     &\leq 3\left(\frac{e}{e-1}\right)^2h(x)\int_0^{(1-\delta) t}n(s<\zeta)\frac{p_{(t-s)/3}^S(0)}{t-s}\,ds + n((1-\delta) t<\zeta)\int_0^\infty q_s^*(x)\,ds\/.
  } 
  Since $t\to p_t^S(0)$ is decreasing (by ($H_1$)), we can write
  \formula{
    \frac{f_t^1(x)}{n(t<\zeta)}\leq \frac{n((1-\delta) t<\zeta)}{n(t<\zeta)}h'(x)+ 
    3\left(\frac{e}{e-1}\right)^2\frac{h(x)p_{\delta t/3}^S(0)}{\delta}\frac{1}{t n(t<\zeta)}\int_0^t n(s<\zeta)\,ds\/.
  }
  Finally, using (\ref{eq:nt:compare}) and the facts that $\lim_{t\to\infty}p_t^S(0)=0$ and $h(x)$ is bounded on $(0,x_0]$, we obtain
  \formula{
    \limsup_{t\to\infty} \frac{f_t^1(x)}{n(t<\zeta)} = (1-\delta)^{-\rho} h'(x)\/.
  }
  Since $\delta$ was arbitrary and $h'$ is bounded on $A$, we get 
  \formula{
    \limsup_{t\to\infty} \frac{f_t^1(x)}{n(t<\zeta)} =  h'(x)\,,
  }
  uniformly in $x\in A$. To deal with $f_t^2(x)$ we use (\ref{eq:gsh:estimate2}) to get
  \formula{
     \frac{f_t^2(x)}{n(t<\zeta)} \leq {\tt d} \frac{e}{e-1}\,p_{t/4}^S(0)\kappa(4/t,0)\,h(x)\,\frac{n(t/2<\zeta)}{n(t<\zeta)}\/.
  }
  Because $h(x)$ is bounded on $A$, 
  \formula{
    \lim_{t\to \infty}\frac{n(t/2<\zeta)}{n(t<\zeta)} = 2^\rho
  }
  and $\lim_{t\to \infty}p_{t/4}^S(0)\kappa(4/t,0) = 0$, we obtain
  \formula{
     \limsup_{t\to\infty} \frac{f_t^2(x)}{n(t<\zeta)} = 0\,,
  }
  uniformly on $A$. Moreover, we have
  \formula{
    \frac{f_t(x)}{n(t<\zeta)}\geq \int_0^tq_s^*(x)ds = h'(x)-\int_t^\infty q_s^{*}(x)\,ds\/,
  }
  where, for $x\in A$, we can write
  \formula{
    \int_t^\infty q_s^{*}(x)\leq 3\left(\frac{e}{e-1}\right)^2 h(x)\int_t^\infty \frac{p_{s/3}^{S}(0)}{s}\,ds<\left(\frac{e}{e-1}\right)^2\sup_{x\in A}h(x)\int_{3t}^\infty \frac{p_{s}^{S}(0)}{s}\,ds\/.
  }
  The last integral goes to zero, when $t$ goes to infinity and consequently, 
  \formula{
     \liminf_{t\to\infty}\frac{f_t(x)}{n(t<\zeta)} = h'(x)\,,
  }
  uniformly in $x\in A$. This ends the proof.
  $\Box$\\
  
\noindent{\it Proof of Theorem $\ref{thm:t:estimates}$}. 
Note that the function 
\formula{
   g(x,t) := \frac{h(x)}{h'(x)}\int_t^{\infty}\frac{q_s^*(x)}{h(x)}\,ds<1
}
is a nonnegative function on $(0,x_0]\times[t_0,\infty]$ such that
\formula{
   c(x_0,t_0) := \sup_{x\leq x_0,t\geq t_0} g(x,t)<1.
}
Since, by (\ref{ptt:infinity}) and (\ref{9255}), the function $g(x,t)$ vanishes when $x$ is small or $t$ is large and $g(x,t)\leq g(x,t_0)$, for $t\geq t_0$, it is enough 
to show that for every $0<a<b$,
\formula{
   \sup_{x\in[a,b]} \frac{1}{h'(x)}\int_{t_0}^\infty g_s^*(x)\,ds <1\/.
}
If the above-given supremum was equal to $1$, then we could choose a sequence of points $(x_n)\in[a,b]$ such that $\lim_{n}x_n=x_0$ and $\lim_{n}g(x_n,t_0)=1$. 
Since, by continuity of $g_s^*(x)$ (see \cite{bib:l12}) together with (\ref{qth:integrability}) and (\ref{ptt:infinity}), the above-given integral is continuous in $x$ we 
would get
\formula{
  \int_{t_0}^\infty g_s^*(x_0)\,ds &= \lim_{n}h'(x_n) = \lim_n\int_0^\infty q_s^*(x_n)\,ds\geq \int_0^\infty \liminf_n q_s^*(x_n)\,ds = \int_0^\infty q_s^*(x_0)\,ds\/.
}
Here we have used the Fatou Lemma and once again continuity of $q_s^*(x)$. Since $q_s^*(x_0)$ is strictly positive this is a contradiction. 

By monotonicity of $n(\cdot<\zeta)$ we can write
\formula{
  f_t(x) &\geq \int_0^t n(s<\zeta)q_{t-s}^*(x)\,ds\geq n(t<\zeta)\int_0^t q_s^*(x)\,ds\\
    &= n(t<\zeta)h'(x)\left(1-\frac{h(x)}{h'(x)}\int_t^{\infty}\frac{q_s^*(x)}{h(x)}\right) \geq c_1  n(t<\zeta)h'(x)
}
whenever $x\leq x_0,t\geq t_0$. Here $c_1 = 1-c(x_0,t_0)>0$.
To deal with the upper-bounds we use the fact that 
\formula{
  n(t<\zeta)\leq \frac{1}{t}\int_0^t n(s<\zeta)\,ds\/.
}
This together with (\ref{qth:integrability}) enable us to write for every $t\geq t_0$ and $x\leq x_0$,
\begin{eqnarray}
  f_t^1(x) &=& \left(\int_0^{t/2}+\int_{t/2}^t\right)n(s<\zeta)q_s^*(x)\,ds\nonumber \\
  &&\leq 3\frac{p_{t/6}(0)}{t}h(x)\int_0^{t/2}n(s<\zeta)\,ds+n(t/2<\zeta)\int_{t/2}^t q_s^*(x)\,ds\nonumber\\
  &&\leq 3{p_{t_0/6}(0)} \frac{h(x)}{h'(x)} h'(x)\frac1t \int_0^t n(s<\zeta)\,ds+\frac{2h'(x)}{t}\int_0^t n(s<\zeta)\,ds\/.\label{5458}
\end{eqnarray}
We deal with the second part $f_t^2(x)$ similarly as in Theorem \ref{thm:t:asymptotic}. We have
\begin{eqnarray}
   f_t^2(x)&\leq& {\tt d}h'(x)\frac{h(x)}{h'(x)} c_{t_0/2}n(t-t_0/2)\nonumber\\
   &\leq& {\tt d} h'(x)\sup_{x\leq x_0}\frac{h(x)}{h'(x)}c_{t_0/2}\frac{1}{t-t_0/2}\int_0^{t-t_0/2}n(s<\zeta)\,ds\nonumber\\
   &\leq& {\tt d} h'(x)\sup_{x\leq x_0}\frac{h(x)}{h'(x)}c_{t_0/2}\frac{2}{t}\int_0^t n(s<\zeta)\,ds\,.\label{5158}
\end{eqnarray}
Inequalities (\ref{5458}) and (\ref{5158}) prove the upper-bounds $f_t\le c_2 n(t<\zeta)h'(x)$, with 
\formula{
  c_2 = 3{p_{t_0/6}(0)} \sup_{x\leq x_0}\frac{h(x)}{h'(x)}+2+2\sup_{x\leq x_0}\frac{h(x)}{h'(x)}c_{t_0/2}\/.
}
The second part of the thesis follows from the first one and (\ref{eq:nt:compare}).
$\Box$\\

\noindent{\it Proof of Proposition $\ref{thm:continuity}$}. 
{$(1)\Rightarrow (2)$} Let $x_0$ be a point of continuity of $h'$. Fix $t>0$ and take $\varepsilon >0$. Since $q_s^*(x)$ 
is continuous, it is enough to show that the integral part $f_t^1(x):=\int_0^t n(t-s<\zeta)q_s^*(x)\,ds$ of (\ref{5519}) is continuous in $x$ at $x_0$. 
Moreover, for every $t_0<t$ we can write
\formula{
 f_t^1(x) = \left(\int_0^{t_0}+\int_{t_0}^t\right)n(t-s<\zeta)q_s^*(x)ds:= k_{t_0}^1(x)+k_{t_0}^2(x)\/.
}
Using the Lebesgue dominated convergence theorem and (\ref{qth:integrability}) we can easily show that $x\to k_{t_0}^2(x)$ is continuous on $(0,\infty)$ for every 
choice of $t_0<t$. Moreover, the same arguments give continuity of the function $x\to \int_{t_0}^\infty q_s^*(x)\,ds$ for every positive $t_0$. We choose $t_0<t/2$ 
such that 
\formula{
  \int_0^{t_0}q_s^*(x_0)\,ds<\frac{\varepsilon}{4n(t/2<\zeta)}\/,
}
where existence of such $t_0$ follows from integrability of $q_s^*(x_0)$ in $s$ at $0$. Since $x\to h'(x)$ is continuous at $x_0$ and the function 
$x\to \int_{t_0}^\infty q_s^*(x)\,ds$ is continuous on $(0,\infty)$, we can choose $\delta>0$ such that for every $|x-x_0|<\delta$,
\formula{
   \int_0^{t_0}q_s^*(x)\,ds<\frac{\varepsilon}{2n(t/2<\zeta)}\/.
}
Writing for $|x-x_0|<\delta$,
\formula{
|f_t^1(x)-f_t^1(x_0)|&\leq n(t-t_0<\zeta)\left(\int_0^{t_0}q_s^*(x)ds+\int_0^{t_0}q_s^*(x_0)ds\right)+\left|k_{t_0}^2(x)-k_{t_0}^2(x_0)\right|\\
&\leq \varepsilon + \left|k_{t_0}^2(x)-k_{t_0}^2(x_0)\right|
}
and taking a limit, when $x\to x_0$ ends the proof in this case.

Since $(3)$ follows directly from $(2)$, it is enough to show $(3) \Rightarrow (1)$. Assume that for some $t>0$ the function $x\to f_t(x)$ is continuous at $x_0$. 
We choose $t_0>0$ such that 
\formula{
  \int_0^{t_0} n(t-s<\zeta)q_s^*(x_0)ds<\varepsilon n(t<\zeta)/4\,,
}
for a given $\varepsilon>0$. Our assumption implies that $x\to k_{t_0}^1(x)$ is continuous at $x_0$ and consequently, we can choose $\delta>0$ such that
\formula{
  \int_0^{t_0} n(t-s<\zeta)q_s^*(x)ds<\varepsilon n(t<\zeta)/2\/,
}
whenever $|x-x_0|<\delta$. Monotonicity of $n(\cdot<\zeta)$ entails,
\formula{
   |h'(x)-h'(x_0)|&\leq \left(\int_0^{t_0}q_s^*(x)ds+\int_0^{t_0}q_s^*(x_0)ds\right)+\left|\int_{t_0}^\infty q_s^*(x)ds-\int_{t_0}^\infty q_s^*(x_0)ds\right|\\
   &\leq \varepsilon+\left|\int_{t_0}^\infty q_s^*(x)ds-\int_{t_0}^\infty q_s^*(x_0)ds\right|\/,
}
whenever $|x-x_0|<\delta$. Since the function $x\to \int_{t_0}^\infty q_s^*(x)ds$ is continuous, the proof is complete.
$\;\;\Box$\\

\subsection*{Acknowledgments}
We would like to thank V\'{\i}ctor Rivero and Ren\'e Schilling for helpful  discussions on the subject of the article and valuable suggestions.

\end{document}